\documentclass{birkjour}
\usepackage{amsmath,amsfonts,amssymb}
\usepackage{microtype}
\usepackage[all,cmtip]{xy}
\usepackage[noadjust]{cite}

\numberwithin{equation}{section}
\newtheorem{thm}{Theorem}[section]  	
			
\newtheorem{cor}[thm]{Corollary}

\newtheorem*{rem*}{Remark}
\theoremstyle{definition}
\newtheorem{dfn}[thm]{Definition}

\newcommand{\R}{\mathbb{R}}
\newcommand{\N}{\mathbb{N}}

\newcommand{\Aa}{{\mathcal A}}
\newcommand{\Cc}{{\mathcal C}}
\def\un{\underline}
\newcommand{\Hh}{{\mathcal H}}
\newcommand{\ba}{\begin{eqnarray}}
	\newcommand{\ea}{\end{eqnarray}}
\newcommand{\bas}{\begin{eqnarray*}}
	\newcommand{\eas}{\end{eqnarray*}}
\newcommand{\be}{\begin{equation}}
	\newcommand{\ee}{\end{equation}}

\newcommand{\conj}[1]{
	\overline{#1} }

\begin{document}
\title[Clifford Algebra-valued Segal-Bargmann and Taylor Isomorphisms]{Clifford Algebra-valued Segal-Bargmann \\Transform and Taylor Isomorphism}
\author[S. Eaknipitsari]{Sorawit Eaknipitsari} 
\address{Department of Mathematics and Computer Science\\
	Faculty of Science, Chulalongkorn University\\
	Bangkok, Thailand} 
\email{e.sorawit@gmail.com}

\author[W. Lewkeeratiyutkul]{Wicharn Lewkeeratiyutkul}
\thanks{The second author is the corresponding author.}

\address{Department of Mathematics and Computer Science\\
	Faculty of Science, Chulalongkorn University\\
	Bangkok, Thailand}
\email{wicharn.l@chula.ac.th}



\date{\today}

\begin{abstract}
	Classical Segal-Bargmann theory studies three Hilbert space unitary isomorphisms that describe the wave-particle duality and the configuration space-phase space. In this work, we generalized these concepts to Clifford algebra-valued functions. We establish the unitary isomorphisms among the space of Clifford algebra-valued square-integrable functions on $\mathbb{R}^n$ with respect to a Gaussian measure, the space of monogenic square-integrable functions on $\mathbb{R}^{n+1}$ with respect to another Gaussian measure and the space of Clifford algebra-valued linear functionals on symmetric tensor elements of $\mathbb{R}^n$. 
\end{abstract}

	\maketitle

\section{Introduction}

For $x \in \R^n$, let $\rho(x) = (2\pi)^{-n/2} \, e^{-|x|^2/2}$. The Segal-Bargmann transform is a map $U \colon L^{2}\left(\mathbb{R}^{n},\rho\,dx\right) \rightarrow \Hh L^2(\mathbb{C}^n, \mu\, dz) $ defined by
	\begin{align*}
		Uf\left(z\right)  & =\int_{\mathbb{R}^{n}}\rho\left(z-x\right) f(x) \,dx\\
		& =\left(2\pi \right)^{-n/2}\int_{\mathbb{R}^{n}}e^{-\frac{|z-x|^{2}}{2}}f(x)  \,dx.
	\end{align*}
Here $\Hh L^2(\mathbb{C}^n, \mu \,  dz)$ is the space of holomorphic square-integrable functions on $\mathbb{C}^n$ with respect to measure $\mu(z)\,dz$ where $\mu(z) =\pi^{-n}\,e^{-|z|^{2}}$ and $dz$ is Lebesgue measure on $\mathbb{C}^n$. Segal \cite{se1}, \cite{se2} and Bargmann \cite{ba} independently proved that $U$ is a unitary isomorphism. See also \cite{ha1}, \cite{hall} for backgrounds and  recent developments.  
The map $U$ can be regarded as the heat operator $e^{\frac{\Delta}2}f = \rho \ast f$, followed by the analytic continuation from $\R^n$ to $\mathbb{C}^n,$ as in the following commutative diagram:
\begin{align}
\begin{gathered}
\xymatrix{
	&&  \Hh L^2(\mathbb{C}^n, \mu \,  dz)   \\
	L^2(\R^n,\rho\, dx)  \ar[rr]_{e^{\frac{\Delta}2}}  \ar[rru]^{U} &&
	\widetilde \Aa (\R^n)
	\ar[u]_{\, \, \Cc \, }} 
\end{gathered}
\end{align}
Here $\mathcal{C}$ denotes the analytic continuation from $\mathbb{R}^n$ to $\mathbb{C}^n$ and $\widetilde \Aa (\R^n)$  is the image of $L^2(\R^n,\rho\, dx)$  by the operator $e^{\frac{\Delta}2}$.
	
There is another space, namely the Fock space $\mathcal{F}(\mathbb{C}^n)$ of symmetric tensors over $\mathbb{C}^n$, that is isometrically isomorphic to $\Hh L^2(\mathbb{C}^n, \mu \,  dz)$. See \cite{ito}, \cite{kaku}, \cite{se3} for original works. Here, we follow recent developments in \cite{d}, \cite{gm}.  Let $X$ be the complex dual space of $\mathbb{C}^n$ and denote by ${X}^{\odot k}$ the space of symmetric $k$-tensors over $X$. Consider the algebraic direct sum $\sum_{k=0}^{\infty} {X}^{\odot k}$ whose elements are of the form
$\alpha = \sum_{k=0}^{\infty} \alpha_k$,
where $\alpha_k \in {X}^{\odot k}$ for each $k$ and $\alpha_k = 0$ for all but finitely many $k$. 
Let $\{ e_1, \dots, e_n\}$ be the standard basis for $\mathbb{C}^n$. 
Each element $\alpha_{k} \in {X}^{\odot k}$ has a natural norm given by
$$|\alpha_k|^2\:  
= \sum_{\substack{|\beta| = k}} \frac{1}{\beta!}|\alpha_k (e^{\beta})|^2,$$
where the sum is taken over multi-indices $\beta = (\beta_1,\dots, \beta_n) \in \mathbb{N}_0^n$. We use notation  $e^{\beta} = e_1^{\odot\beta_1} {\odot} \dots {\odot}e_n^{\odot\beta_n}$, $|\beta| = \beta_1 + \dots + \beta_n$ and $\beta! = \beta_1!\dots\beta_n!$.
The algebraic direct sum $\sum_{k=0}^{\infty} {X}^{\odot k}$ is equipped with a norm given by
$$\|\alpha\| = \big(\sum_{k=0}^{\infty}  |\alpha_k|^2\big)^{\frac{1}{2}}.$$	
The Fock space $\mathcal{F}(\mathbb{C}^n)$ is defined to be the Hilbert space completion of the algebraic direct sum with respect to this norm. Thus $\mathcal{F}(\mathbb{C}^n)$ is the set of strong sums $\sum_{k=0}^{\infty} \alpha_k$, where $\alpha_k \in {X}^{\odot k}$ for each $k$, such that $\sum_{k=0}^{\infty}  |\alpha_k|^2 < \infty$.

Next, we describe the unitary isomorphism from $\Hh L^2(\mathbb{C}^n, \mu\, dz)$ onto $\mathcal{F}(\mathbb{C}^n)$. 
Let $f$ be a holomorphic function on $\mathbb{C}^{n}$. There is a linear map $D^kf \colon(\mathbb{C}^n)^{\odot k} \rightarrow \mathbb{C}$ such that for any $u_1,\dots,u_k \in \mathbb{C}^n,$
$$D^kf (u_1\odot\dots\odot u_k ) =
\partial_{u_1}\dots \partial_{u_k}f(0), $$
with $D^0f = f(0).$ Here $\partial_v$ is the directional derivative in the $v$ direction. We identify $D^k f$ as an element of $(X)^{\odot k}$. 
It is natural to write $\sum_{k=0}^{\infty} D^kf$ in $\mathcal{F}(\mathbb{C}^n)$ as $$(1-D)^{-1}f = \sum_{k=0}^{\infty} D^kf.$$

The map $(1-D)^{-1}$ is a unitary isomorphism from  $\Hh L^2(\mathbb{C}^n, \mu\, dz)$ onto the Fock space $\mathcal{F}(\mathbb{C}^n).$ This map is simply the Taylor series expansion that assigns to each holomorphic function its Taylor coefficients. Hence we will call it the \emph{Taylor map}.
	
It can be summarized that the three arrows in the following commutative diagram are unitary isomorphisms. These isomorphisms are used to describe the ``wave-particle duality" in quantum field theory.
	\begin{align}\label{utriad}
	\begin{gathered}
	\xymatrix{
		L^2(\mathbb{R}^n,\rho \,dx) \ar[rr]^{U} \ar[rd]_{}	&	 	&	\Hh L^2(\mathbb{C}^n, \mu\, dz) \ar[ld]^{(1-D)^{-1}}\\
		&	\mathcal{F}(\mathbb{C}^n)	&} 
	\end{gathered}
	\end{align}

 There is another form of a Segal-Bargmann transform $V \colon L^2(\mathbb{R}^n, dx) \rightarrow \Hh L^2(\mathbb{C}^n, \nu\, dx\,dy)$ defined by
 \bas
 Vf(z) &=& \int_{\R^n} \, \rho(z - x) f(x) \, dx \\
 &=& (2\pi )^{-n/2} \, \int_{\R^n}  \, e^{-\frac{|z-x|^2}{2}}\,f(x)\, dx ,
 \eas
 where $\nu(y)\,dx\,dy ={\pi}^{-n/2} \,e^{-|y|^2}dx\,dy$.  
 The map $V$ is a unitary isomorphism from $L^2(\mathbb{R}^n, dx)$ onto $\Hh L^2(\mathbb{C}^n, \nu\, dx\,dy)$. The formula that defines $V$ is the same as that for $U$, but with different domain and range. However, one does not have a Taylor map to the Fock space in the same way as the $U$-version of a Segal-Bargmann transform. The main reason is that the monomials are orthogonal with respect to the measure $d\mu$ and not to the measure $d\nu$. The $U$-version and the $V$-version both have their advantages, but certainly the existence of this Taylor map onto the Fock space is a significant advantage of the $U$-version.


The purpose of this paper is to generalize the triad (\ref{utriad}) to the Clifford algebra-valued functions setting. 
In 1982, Brackx, Delanghe, and Sommen \cite{bds} defined a (left) monogenic function $f : \mathbb{R}^n \rightarrow \mathbb{C}_n$ as an element in the kernel of a Dirac operator, i.e., $$ 0 = {\underline{D}}\,f(\underline{x}) = \sum_{j=1}^{n} e_j\partial_{e_j}f(\underline{x}) $$
where $\mathbb{C}_n$ is the complex Clifford algebra generated by the standard basis $\{e_1,\dots,e_n\}$ of $\mathbb{R}^n$. 
Denote by $\mathcal{M}(\mathbb{R}^n)$ and $\mathcal{M}(\mathbb{C}^n)$ the right $\mathbb{C}_n$-modules of monogenic functions on $\mathbb{R}^n$ and $\mathbb{C}^n$, respectively. In 2016, Kirwin, Mour{\~a}o, Nunes, and Qian \cite{kmnq} used a notion of an $(n+1)$-variable monogenic function, namely a function $f : \mathbb{R}^{n+1} \rightarrow \mathbb{C}_n$ such that 
$$ (\partial_{e_0} + {\underline{D}}\,)f(x_0,\underline{x}) = 0.$$ They obtained a generalized Segal-Bargmann transform on special types of monogenic functions namely, slice monogenic and axial monogenic functions. In 2017, Mour{\~a}o, Nunes, and Qian \cite{mnq} continued their work and  generalized the Segal-Bargmann transform to Clifford algebra-valued functions analogous to $V$ as in the following theorem. 
	
\begin{thm}[\cite{mnq}]\label{vextend}
The map $\tilde{V} \colon L^2(\mathbb{R}^n,d\underline{x})\otimes \mathbb{C}_n \rightarrow \mathcal{M}L^2(\mathbb{R}^{n+1},\tilde{\nu}\, dx_0\,d\underline{x})$ given by
$$\tilde{V}(f)(x_0,\underline{x}) = (2\pi)^{-n}\int_{\mathbb{R}^n}\left( \int_{\mathbb{R}^n} e^{-\frac{\underline{p}^2}{2}} e^{i(\underline{p},\underline{x}-\underline{y})}e^{-ix_0\underline{p}}\,d\underline{p}\right) f(\underline{y})\,d\underline{y},$$
is a unitary isomorphism. Here $\mathcal{M}L^2(\mathbb{R}^{n+1},\tilde{\nu} \, dx_0\,d\underline{x})$ is the Hilbert space of monogenic functions on $\mathbb{R}^{n+1}$ that are square-integrable with respect to measure $\tilde{\nu} \, dx_0\,d\underline{x}$ where $\tilde{\nu}(x_0) = \frac{1}{\sqrt{\pi}}e^{-x_0^2}$.
	\end{thm}   
In 2018, Dang, Mour{\~a}o, Nunes, and Qian \cite{dmnq} also obtained this result for spherical domains. 
Using the idea of Theorem \ref{vextend} in \cite{mnq}, we can extend the unitary map $U$ in the classical setting as follows.
	
\begin{thm} \label{main0}
The map $\tilde{U}$  given by 
\[
\tilde{U}(f)(x_0,\underline{x}) = (2\pi)^{-n}\int_{\mathbb{R}^n}\left( \int_{\mathbb{R}^n} e^{-\frac{\underline{p}^2}{2}} e^{i(\underline{p},\underline{x}-\underline{y})}e^{-ix_0\underline{p}}\,d\underline{p}\right) f(\underline{y})\,d\underline{y},
\]		
is a unitary isomorphism from $L^2(\mathbb{R}^n,\rho \,d\underline{x})\otimes \mathbb{C}_n$ onto  $\mathcal{M}L^2(\mathbb{R}^{n+1},d\tilde{\mu})$, the Hilbert space of monogenic functions on $\mathbb{R}^{n+1}$ that are square-integrable with respect to the measure $$ d\tilde{\mu} = \frac{1}{\pi^{(n+1)/2}}e^{-x_0^2-|\underline{x}|^2}\,dx_0\,d\underline{x}.$$
\end{thm}
The map $\tilde{U}$ can be factorized as in the following diagram:
\begin{align}
\begin{gathered}
\xymatrix{
		&&  \mathcal{M}L^2 (\R^{n+1}, \,d\tilde{\mu} )  \\
		L^2(\R^n,\rho \, d \un x)\otimes \mathbb{C}_{n}  \ar@{->}[rr]_{e^{\frac{\Delta}2}}  \ar[rru]^{\tilde{U}} && \widetilde \Aa (\R^n)\otimes 
		\mathbb{C}_{n}
		\ar[u]_{\, e^{-x_0 \un D}}
	}
\end{gathered}
\end{align}
Here we replace the analytic continuation $\mathcal{C}$ by the Cauchy-Kowalevski extension $e^{-x_0\underline{D}}$, which will be explained in Section 3.
	
Now we turn to the Clifford algebra-valued Fock space. Let $X$ be the real dual space of $\R^n$.
We can repeat the construction in the classical case for the Clifford algebra-valued symmetric tensor algebra, which will be identified with $\mathcal{F}(X)\otimes \mathbb{C}_n$ and called the \emph{$\mathbb{C}_n$-valued covariant Fock space}. An element in $\mathcal{F}(X)\otimes \mathbb{C}_n$ is a strong sum $\sum_{k=0}^{\infty} \alpha_k$, where each $\alpha_{k} \in {X}^{\odot k} \otimes \mathbb{C}_n$ and such that 
\[
\|\alpha\|^2 = \sum_{k=0}^{\infty}  |\alpha_k|^2 < \infty.
\]	

Let $f \colon \mathbb{R}^{n+1} \rightarrow \mathbb{C}_n$ be a monogenic function. Then there is a linear map $D^kf \colon(\mathbb{R}^n)^{\odot k} \rightarrow \mathbb{C}_n$ such that for any $u_1,\dots,u_k \in \mathbb{R}^n,$
$$D^kf (u_1\odot\dots\odot u_k ) =
\partial_{u_1}\dots \partial_{u_k}f(0, \underline{0}), $$
with $D^0f = f(0,\underline{0}).$ It is natural to write $\sum_{k=0}^{\infty} D^kf \in \mathcal{F}(X)\otimes\mathbb{C}_n$ as 
\[
   (1-D)^{-1}f = \sum_{k=0}^{\infty} D^kf,
\]
where
\[ 
\|(1-D)^{-1} f\|^2 =  \sum_{k=0}^{\infty}  |D^kf|^2 < \infty.
\]

Here is the second main theorem in this paper. 

\begin{thm} \label{main02}
 The map $(1-D)^{-1}$ is a unitary isomorphism from the space of square-integrable monogenic functions   $\mathcal{M}L^2(\mathbb{R}^{n+1},\,d\tilde{\mu})$ onto the Clifford algebra-valued Fock space $\mathcal{F}(X)\otimes \mathbb{C}_n.$ 
\end{thm}
	
Combining Theorems \ref{main0} and \ref{main02} together, we have the following unitary isomorphisms in the Clifford-algebra-valued setting in this diagram: 	

\begin{align} 
\begin{gathered}
\xymatrix{
	L^2(\mathbb{R}^n,\rho \,d\underline{x})\otimes \mathbb{C}_n \ar[rr]^{\tilde{U}} \ar[rd]_{}	&	 	&	\mathcal{M}L^2(\mathbb{R}^{n+1},\,d\tilde{\mu}) \ar[ld]^{(1-D)^{-1}}\\
	&	\mathcal{F}(X)\otimes \mathbb{C}_n	&} \notag
\end{gathered}
\end{align}

The paper is organized as follows. In Section 2, we recall necessary facts about Clifford algebra and Clifford analysis used in this paper. In Section 3, we discuss the Clifford algebra-valued Segal-Bargmann Transform and prove Theorem \ref{main0}. In Section 4, we discuss the Clifford algebra-valued Fock space in detail and prove Theorem \ref{main02}.

\section{Preliminaries}
	
\subsection{Real and Complex Clifford algebras}

Let $\mathbb{K} = \mathbb{R}$ or $\mathbb{C}$. Define the Clifford algebra $\mathbb{K}_n$ as the $\mathbb{K}-$algebra generated by $n$ elements $e_1,\dots,e_n$, which can be identified with the canonical basis of $\mathbb{K}^n \subset \mathbb{K}_n$ and satisfy the relations $e_ie_j + e_je_i = -2\delta_{ij}$, see e.g.\cite{bds},  \cite{bss}, \cite{sg}, \cite{mnq}. If $\mathbb{K} = \mathbb{R}$ we call $\mathbb{K}_n$ the real Clifford algebra and if $\mathbb{K} = \mathbb{C}$ we call $\mathbb{K}_n$ the complex Clifford algebra. 

Note that $\{e_A ~|~ A \subset \{1,2,\dots,n\} = N \}$ is a basis for $\mathbb{K}_n$ where  $e_A = {e}_{i_1}{e}_{i_2}\cdots {e}_{i_k}$ with $A =\{i_1,i_2,\dots,i_k\}$, $1\leq i_1< i_2<\cdots<i_k \leq n$, and ${e}_{\varnothing} =1$.
Thus any $\lambda \in \mathbb{K}_n$ can be written as 
$$\lambda = \sum_{A\subset N} \lambda_A\,e_A,$$ 
where $\lambda_A \in \mathbb{K}$.
Define the so-called \textbf{$k$-vector part} of $\lambda$, for $k = 0,1,\dots,n$, by 
$$[\lambda]_k = \sum_{|A| = k} \lambda_A\,e_A.$$
Now, we focus at $\mathbb{C}_n$.
One important operator of $\mathbb{C}_n$, the \textbf{Hermitian conjugation}, is defined by 
\begin{align*}
		&\conj{e_i} = -e_i, \quad i=1,2,\dots,n,\\
		&\conj{(\lambda_A e_A)} =\lambda_A^c\, \conj{e_A}, \quad \lambda_A\in \mathbb{C}, A \subset N, \\
		&\conj{(\lambda\mu)} = \conj{\mu}\, \conj{\lambda}, \quad \lambda, \mu \in \mathbb{C}_n,
\end{align*}
where $\lambda_A^c$ denotes the complex conjugate of the complex number $\lambda_A$.
This contributes to a Hermitian inner product and its associated norm on $\mathbb{C}_n$, respectively defined by 
$$ (\lambda,\mu) = [\conj{\lambda}\mu]_0 \quad \text{and} \quad |\lambda|^2 = [\conj{\lambda}\lambda]_0 = \sum_A|\lambda_A|^2.$$

\subsection{Clifford analysis}
Clifford analysis is a function theory in higher dimensions generalizing complex analysis, see e.g. \cite{bds}. We begin by considering the generalized Cauchy-Riemann operator
$$ \partial_{e_{0}} + \underline{D},$$ 
where $$\underline{D} = \sum_{j=1}^{n} e_j\partial_{e_j}\,$$
and $\{e_0,e_1,\dots,e_n\}$ is the standard basis of $\mathbb{R}^{n+1}$.
To make things easier, we also identify $\mathbb{R}^n$ with the subspace of $\mathbb{R}_n$ of 1-vectors 
\[ 
	\{\underline{x} = \sum_{j=1}^{n} x_je_j : x = (x_1,\dots,x_n) \in \mathbb{R}^n\}. 
\] 

Now we give a generalized concept of a holomorphic function. A continuously differentiable function $f$ on an open domain $\mathcal{O} \subset \mathbb{R}^{n+1}$, taking values in $\mathbb{C}_n$, is called  (left) \textbf{monogenic} on $\mathcal{O}$ if it satisfies the generalized Cauchy-Riemann equation: 
$$ (\partial_{e_{0}} + \underline{D})f(x_0,\underline{x}) = 0.$$ 

\begin{thm}[\cite{bds}]\label{monogenicexist}
	Let $f$ be a $\mathbb{C}_n$-valued analytic function on $\R^n$. Then there exists a unique $\mathbb{C}_n$-valued monogenic function $F$ on $\R^{n+1}$ such that $F(0,\underline{x}) = f(\underline{x})$. 
\end{thm}
This extension is called the Cauchy-Kowalevski extension, or simply the C-K extension, of $f$. In \cite{sf} and \cite{dss}, the formula for the C-K extension is given as follows:

\begin{thm} \label{2.3}
Let $f$ be a $\mathbb{C}_n$-valued analytic function on $\R^n$.
Then the C-K extension of $f$ is given by the formula
$$F(x_0,\underline{x})= e^{-x_0\underline{D}}f(\underline{x}) := \left( \sum_{k=0}^{\infty} (-1)^k \dfrac{x_0^k}{k!}\underline{D}^kf\right) (\underline{x}) $$
	where the series converges uniformly on compact subsets.
\end{thm}

\section{Clifford algebra-valued Segal-Bargmann Transform}

We introduce the Hilbert space of Clifford algebra-valued square-integrable functions with respect to measure $\rho$ on $\mathbb{R}^{n}$ 
\[
L^2(\R^{n}, d\rho\,;\,\mathbb{C}_n) =\{f : \R^{n} \rightarrow \mathbb{C}_n ~|~ \int_{\mathbb{R}^{n}}  |f(\underline{x})|^2 \,\rho(x)\,dx < \infty \},
\] 
equipped with the inner product:	
$$ \langle f,g \rangle = \int_{\mathbb{R}^{n}} \left(f(\underline{x}),g(\underline{x})\right) \rho(x)\,dx =\int_{\mathbb{R}^{n}} [\,\conj{f(\underline{x})}g(\underline{x})\,]_0\, \rho(x)\,dx$$
where
\[ \rho (\underline{x}) = (2\pi)^{-n/2}\, e^{-|\underline{x}|^2/2}.\]

We identify $L^2(\R^{n}, d\rho\,;\,\mathbb{C}_n)$ with the tensor product $L^2(\R^{n}, d\rho) \otimes \mathbb{C}_n$.
Also, the Hilbert space of Clifford algebra-valued square-integrable functions with respect to measure $\tilde{\mu}$ on $\mathbb{R}^{n+1}$ is given by
\[
L^2(\R^{n+1}, d\tilde{\mu}\,;\,\mathbb{C}_n) =\{F : \R^{n+1} \rightarrow \mathbb{C}_n ~|~ \int_{\mathbb{R}^{n+1}}  |F(x)|^2 \,d\tilde{\mu} < \infty \},
\] 
equipped with the inner product:	
$$ \langle F,G \rangle = \int_{\mathbb{R}^{n+1}} \left(F(x),G(x)\right) d\tilde{\mu} =\int_{\mathbb{R}^{n+1}} [\,\conj{F(x)}G(x)\,]_0\, d\tilde{\mu}$$
where
$$ d\tilde{\mu} = \pi^{-(n+1)/2}e^{-x_0^2-|\underline{x}|^2}\,dx_0\,d\underline{x}.$$

The space $L^2(\R^{n+1}, d\tilde{\mu}\,;\,\mathbb{C}_n)$ can be identified with the tensor product
$ L^2(\R^{n+1}, d\tilde{\mu})\otimes\mathbb{C}_n$.
Denote by $\mathcal{M}L^2(\mathbb{R}^{n+1},d\tilde{\mu})$ the Hilbert space of monogenic functions on $\mathbb{R}^{n+1}$ that are square-integrable with respect to measure $\tilde{\mu}$.

Following the idea of the proof of Theorem \ref{vextend} in \cite{mnq}, we will prove Theorem \ref{main0}, namely,  $\tilde{U}$ is a unitary isomorphism from $L^2(\mathbb{R}^n,d\rho)\otimes \mathbb{C}_n$ onto  $\mathcal{M}L^2(\mathbb{R}^{n+1},d\tilde{\mu})$.
The map $\tilde{U}$ is the heat operator applied to an element in the domain and then followed by the C-K extension. 
In other words, 
\[
\tilde{U}(f) 
= \big(e^{-x_0 \underline{D}} \circ e^{\frac{\triangle}{2}}\big)(f) 
= e^{-x_0 \underline{D}} (\rho \ast f). 
\]
Note that $\rho$ is analytic on $\R^n$, and so is $\rho \ast f$. Then its C-K extension, $\tilde{U}(f)$, exists and is monogenic on $\mathbb{R}^{n+1}$.

\begin{proof}[Proof of Isometry]
Note that the Schwartz space of $\mathbb{C}_n$-valued functions is identified with the tensor product $\mathcal{S}(\mathbb{R}^n)\otimes \mathbb{C}_n$. Since $\mathcal{S}(\mathbb{R}^n)$ is dense in $L^2(\mathbb{R}^n,d\rho)$, it follows that  $\mathcal{S}(\mathbb{R}^n)\otimes \mathbb{C}_n$ is dense in $L^2(\mathbb{R}^n,d\rho)\otimes \mathbb{C}_n$.
Any $f\in \mathcal{S}(\mathbb{R}^n)\otimes \mathbb{C}_n$ can be written as $f = \sum_{A\subset N} f_A e_A$, where $f_A \in \mathcal{S}(\mathbb{R}^n)$. 
Hence the Fourier transform of $f$ is given by $\widehat{f} = \sum_{A\subset N} \widehat{f_A}\,e_A.$  
By the density argument, it suffices to show that $\tilde{U}$ is an isometry on $\mathcal{S}(\mathbb{R}^n)\otimes \mathbb{C}_n$. 

The Fourier inversion formula of $f \in \mathcal{S}(\mathbb{R}^n)\otimes \mathbb{C}_n$ is
\begin{equation}\label{Fourier_inversion}
f(x) = \frac{1}{(2\pi)^{n/2}}\int_{\mathbb{R}^n}e^{i(\underline{p},\underline{x})} \hat{f}(\underline{p})\,d\underline{p}.
\end{equation}
By applying the operator $e^{-x_0 \underline{D}} \circ e^{\frac{\triangle}{2}}$ to $f$ in (\ref{Fourier_inversion}) and pass it inside the integral sign, we see that
\begin{equation}\label{second}
\tilde{U}(f)(x_0,\underline{x}) \;
= \; \frac{1}{(2\pi)^{n/2}}\int_{\mathbb{R}^n}e^{-ix_0\underline{p}}\,e^{-\frac{|\underline{p}|^2}{2}}\,e^{i(\underline{p},\underline{x})}\, \hat{f}(\underline{p})\,d\underline{p}.
\end{equation}
Note that since $f$ and $\hat{f}$, as well as their derivatives, are rapidly decaying (i.e. they are functions in the Schwarz space), this allows passage of the operator inside the integral sign and the interchange of the order of integration.

To show the isometry of $\tilde{U}$, let $f, h\in \mathcal{S}(\mathbb{R}^n)\otimes \mathbb{C}_n$. Then
\begin{align*}
\langle f,h\rangle 
&=\frac{1}{(2\pi)^{n/2}} \int_{\mathbb{R}^{n}}   \big( f(\underline{x}),h(\underline{x}) \big) e^{-\frac{|\underline{x}|^2}{2}}d\underline{x}\\ 
&=\frac{1}{(2\pi)^{3n/2}} \int_{\mathbb{R}^{n}}\int_{ \mathbb{R}^{2n}} \big(e^{i(\underline{p},\underline{x})}  \hat{f}(\underline{p}),e^{i(\underline{q},\underline{x})} \hat{h}(\underline{q}) \big) e^{-\frac{|\underline{x}|^2}{2}}\,d\underline{p}\,d\underline{q}\,d\underline{x} \\
&=\frac{1}{(2\pi)^{3n/2}} \int_{\mathbb{R}^{2n}}\big( \hat{f}(\underline{p}),\hat{h}(\underline{q}) \big)\int_{ \mathbb{R}^n} e^{i(\underline{q}-\underline{p},\underline{x})}   e^{-\frac{|\underline{x}|^2}{2}}\,d\underline{x}\,d\underline{p}\,d\underline{q} \\
&= \frac{1}{(2\pi)^n} \int_{\mathbb{R}^{2n}} e^{-\frac{|\underline{q}-\underline{p}|^2}{2}} \big( \hat{f}(\underline{p}),\hat{h}(\underline{q}) \big) d\underline{p}\,d\underline{q}. 
\end{align*}
The last equality is obtained by the Fourier transform.
On the other hand,
\begin{equation}\label{Eq:inner product}
\big\langle \tilde{U}(f), \tilde{U}(h) \big\rangle 
\; =  \; \frac{1}{\pi^{\frac{n+1}{2}}}\int_{\mathbb{R}^{n+1}}\big(\tilde{U}(f), \tilde{U}(h)\big) e^{-x^2_0}e^{-|\underline{x}|^2}\,dx_0\,d\underline{x}
\end{equation}
Note that $\big(e^{-ix_0\underline{p}}\hat{f}(\underline{p}),e^{-ix_0\underline{q}}\hat{h}(\underline{q})\big) = e^{-ix_0(\underline{p}+\underline{q})}\big(\hat{f}(\underline{p}),\hat{h}(\underline{q})\big)$ because $\underline{p}$ is a 1-vector in $\mathbb{R}_n$, which implies $\conj{i\underline{p}} = i\underline{p}$.  
It follows from (\ref{second}) that 
\begin{equation*}\label{third}
\big(\tilde{U}(f), \tilde{U}(h)\big)  = \frac{1}{(2\pi)^n}\int_{\mathbb{R}^{2n}}
e^{i(\underline{q}-\underline{p},\underline{x})}e^{-\frac{|\underline{p}|^2+|\underline{q}|^2}{2}} e^{-ix_0(\underline{p}+\underline{q})}\big(\hat{f}(\underline{p}),\hat{h}(\underline{q})\big) \,d\underline{p}\,d\underline{q}.
\end{equation*}
Substitute 
this into (\ref{Eq:inner product}) and interchange the order of integration by Fubini's theorem. Let us calculate the integral with respect to $d\underline{x}$ and $dx_0$ first. The  Fourier transform yields the following integrals
\begin{align}
\int_{\R^n} e^{i(\underline{q}-\underline{p},\underline{x})}e^{-|\underline{x}|^2}\,d\underline{x}  \; &= \; \pi^{n/2}e^{-\frac{|\underline{q}-\underline{p}|^2}{4}}. \label{first} \\  
\int_{\mathbb{R}} e^{-ix_0(\underline{p}+\underline{q})}e^{-x_0^2}\,dx_0 &=
\sqrt{\pi} e^{\frac{|\underline{p}+\underline{q}|^2}{4}} \label{EQ:integrate wrt x_0}
\end{align}
Putting (\ref{first}) and (\ref{EQ:integrate wrt x_0}) in (\ref{Eq:inner product}) and applying the Parallelogram Law, we have 
\begin{align*}
\big\langle \tilde{U}(f),\tilde{U}(h) \big\rangle 
	&= \frac{1}{(2\pi)^n}\int_{\mathbb{R}^{2n}} e^{\frac{|\underline{p}+\underline{q}|^2-|\underline{q}-\underline{p}|^2}{4}} e^{-\frac{|\underline{p}|^2+|\underline{q}|^2}{2}}
	\big( \hat{f}(\underline{p}),\hat{h}(\underline{q})\big)\,d\underline{p}\,d\underline{q} \\
	&= \frac{1}{(2\pi)^n} \int_{\mathbb{R}^{2n}} e^{-\frac{|\underline{q}-\underline{p}|^2}{2}} \big( \hat{f}(\underline{p}),\hat{h}(\underline{q}) \big) d\underline{p}\,d\underline{q}.
\end{align*}
This establishes the isometry part of $\tilde{U}$.
\phantom{\qedhere}
\end{proof}
\begin{proof}[Proof of Surjectivity]
Let $\{H_{\beta} : \beta \in \mathbb{N}_0^n \}$ denote the orthogonal basis of $L^2(\mathbb{R}^n,d\rho)$ consisting of $n$-dimensional Hermite polynomials, each of which is a product of $1$-dimensional Hermite polynomials, i.e. for
\[
H_{\beta}(\underline{x})= H_{\beta_1}(x_1)\dots H_{\beta_n}(x_n)
\]  
for  $\beta = (\beta_1,\dots,\beta_n) \in \mathbb{N}_0^n.$
Morever, $\|H_{\beta}\|^2 = \beta! = \beta_1!\dots\beta_n!$.
Details about Hermite polynomials can be found in standard literatures, e.g. \cite{dbgs}.
It can be directly computed that 
\begin{equation}\label{heathermite}
e^{\frac{\Delta}{2}}H_{\beta} = \underline{x}^\beta  = x_1^{\beta_1}\dots x_n^{\beta_n}.
\end{equation} 
	
Let $G \in \mathcal{M}L^2(\mathbb{R}^{n+1},d\tilde{\mu})$. Then $g(\underline{x}) =  G(0,\underline{x})$ is an analytic function  and hence it has a Taylor expansion with infinite radius of convergence
\[
g(\underline{x}) \; = \; \sum_A\sum_{\beta\in\mathbb{N}_0^n} \alpha_{\beta,A} \,\underline{x}^\beta\, e_A.
\]
Take  
\[
f = \sum_A\sum_{\beta\in\mathbb{N}_0^n} \alpha_{\beta,A} \,H_\beta\, e_A.
\]	where the series is taken in the $L^2$-norm sense. Then
	$f \in L^2(\mathbb{R}^n,d\rho)\otimes\mathbb{C}_n$.	
Since $\tilde{U}$ is bounded, we can pass it inside the summations: 
\[
\tilde{U}(f) = \sum_A\sum_{\beta\in\mathbb{N}_0^n} \alpha_{\beta,A} \,\tilde{U}(H_\beta)\, e_A.
\]
Since $\tilde{U}(f)$ is monogenic, we can evaluate its value at $x_0 = 0$. 
\begin{align*}
	\tilde{U}(f)(0,\underline{x}) 
	= \sum_A\sum_{\beta\in\mathbb{N}_0^n} \alpha_{\beta,A} \,e^{\frac{\triangle}{2}}(H_\beta)\, e_A 
	= \sum_A\sum_{\beta\in\mathbb{N}_0^n} \alpha_{\beta,A} \,\underline{x}^\beta\, e_A.
\end{align*} 
Hence $\tilde{U}(f)(0,\underline{x}) = g(\underline{x})$, which implies $\tilde{U}(f) = G.$
We have established that $\tilde{U}$ is a unitary map from $L^2(\mathbb{R}^n,d\rho)\otimes\mathbb{C}_n$ onto $\mathcal{M}L^2(\mathbb{R}^{n+1}, d\tilde{\mu})$.
\end{proof}

Since $\tilde{U}$ is a unitary map, it follows that $\{\tilde{U}(H_\beta) :  \beta \in \mathbb{N}_0^n\}$ is an orthogonal basis for $\mathcal{M}L^2(\mathbb{R}^{n+1}, d\tilde{\mu})$. Define
\begin{equation} \label{Def of P_beta}
P_\beta(x_0,\underline{x}) = \tilde{U}(H_\beta)(x_0,\underline{x}) = e^{-x_0\underline{D}}\,\underline{x}^\beta. 
\end{equation}

Moreover, $\|P_\beta\|^2 = \|\tilde{U}(H_\beta)\|^2 = \|H_\beta\|^2 = \beta!$.
We put it into the following Corollary.
	
\begin{cor}\label{onb}
$\{P_\beta : \beta \in \mathbb{N}_0^n\}$ is an orthogonal basis for $\mathcal{M}L^2(\mathbb{R}^{n+1}, d\tilde{\mu})$ and $\|P_\beta\|^2 = \beta!$ for each $\beta \in \mathbb{N}_0^n$.
\end{cor}

\section{Clifford algebra-valued Fock space}

Let $X=(\mathbb{R}^n)^*$, the real dual space of $\R^n$. Denote by $X^{\odot k}$ the algebraic symmetric $k$-tensor product of $X$. We will write $\text{Sym}(X)$ for the algebraic symmetric tensor algebra over $X$, i.e.\ $\text{Sym}(X)$ is the weak direct sum $\sum_{k=0}^\infty X^{\odot k}$ consisting of elements of the form $\sum_{k=0}^{\infty} \alpha_k$, where each $\alpha_{k} \in {X}^{\odot k}$ and $\alpha_k = 0$ for all but finitely many $k$. Each $\alpha_{k} \in {X}^{\odot k}$ has a natural norm given by
$$|\alpha_k|^2 = \sum_{\substack{0\leq \beta_1,\dots, \beta_n \leq k\\ \beta_1 + \dots + \beta_n = k}} \frac{1}{\beta_1!\dots\beta_n!}|\alpha_k(e_1^{\odot\beta_1} {\odot} \dots {\odot}e_n^{\odot\beta_n})|^2  = \sum_{\substack{|\beta| = k\\	\beta \in \mathbb{N}_0^n}} \frac{1}{\beta!}|\alpha_k (e^{\beta})|^2,$$
where  $\{ e_1, \dots, e_n\}$ is the standard basis for $\R^n$. We define $\mathcal{F}(X)$ to be the Hilbert space completion of $\text{Sym}(X)$ with respect to the norm
\[
\|\alpha\| = \big(\sum_{k=0}^{\infty}  |\alpha_k|^2\big)^{\frac{1}{2}}
\]
and call it the \textbf{covariant Fock space}.
		
We can repeat the construction above for the Clifford algebra-valued symmetric tensor algebra, which will be identified with $\mathcal{F}(X)\otimes \mathbb{C}_n$ and called the \textbf{$\mathbb{C}_n$-valued covariant Fock space}. An element in $\mathcal{F}(X)\otimes \mathbb{C}_n$ is a strong sum $\sum_{k=0}^{\infty} \alpha_k$, where each $\alpha_{k} \in {X}^{\odot k} \otimes \mathbb{C}_n$ and such that 
		\[
		  \|\alpha\|^2 = \sum_{k=0}^{\infty}  |\alpha_k|^2 < \infty.
		\]	

\begin{dfn}
Let $f: \mathbb{R}^{n+1} \rightarrow \mathbb{C}_n$ be a monogenic function. Let $\{ e_1, \dots, e_n\}$ be the standard basis of $\R^n$.  For each $(x_0,\underline{x}) \in \mathbb{R}^{n+1}$, define the directional derivative operator $Df(x_0, \underline{x})$ to be a linear map on $\R^n$ such that $Df(x_0,\underline{x})(e_i) = \partial_{e_i}f(x_0,\underline{x})$ for each $i$. More generally, for each $k \in \mathbb{N}$, define a $k$-linear map $D^k f(x_0,\underline{x})$ on $(\R^n)^k$ by
		\[D^kf(x_0,\underline{x})(e_{i_1},\dots,e_{i_k}) =\partial_{e_{i_1}}\dots \partial_{e_{i_k}}f(x_0, \underline{x}).\]
\end{dfn}
Since $f$ is smooth, $D^kf(x_0,\underline{x})$  is a symmetric $k$-linear map on $(\R^n)^k$.
This induces a $\mathbb{C}_n$-valued linear map on $(\R^n)^{\odot k}$ which is also denoted by $D^kf(x_0,\underline{x})$. Thus for any $u_1,\dots,u_l \in \R^n$,
\[ D^kf(x_0,\underline{x}) (u_1\odot\dots\odot u_l ) =\begin{cases}
\partial_{u_1}\dots \partial_{u_k}f(x_0, \underline{x}) \quad &\text{if } l=k; \\
0 &\text{otherwise.}
\end{cases} \]
Note that, for each $k\in \N_0$, the norm of $D^kf(x_0,\underline{x})$ is
\begin{align*}
		|D^kf(x_0,\underline{x})|^2 
		&=  \sum_{\substack{|\beta| = k\\	\beta \in \mathbb{N}_0^n}} \frac{1}{\beta!}
		|\partial^{\beta}f(x_0,\underline{x})|^2,
\end{align*}
where $\partial^{\beta} = \partial_1^{\beta_1}\dots\partial_n^{\beta_n}$ for $\beta = (\beta_1,\dots,\beta_n)$.
Next, we identify $D^kf(x_0,\underline{x})$ as an element of $(X)^{\odot k}\otimes \mathbb{C}_n$. With $D^0f(x_0,\underline{x})$ defined as $f(x_0,\underline{x})$, it is natural to write
\[ (1-D)_{(x_0,\underline{x})}^{-1} f = \sum_{k=0}^{\infty} D^kf(x_0,\underline{x}).\]
Then $(1-D)_{(x_0,\underline{x})}^{-1} f \in \mathcal{F}(X) \otimes \mathbb{C}_n$ if
\[ \|(1-D)_{(x_0,\underline{x})}^{-1} f\|^2 :=  \sum_{k=0}^{\infty}  |D^kf(x_0,\underline{x})|^2 < \infty.\]
		
For simplicity, we write $(1-D)^{-1}$ instead of $(1-D)_{(0,\underline{0})}^{-1}$. We will prove Theorem \ref{main02}, which states that the map $(1-D)^{-1}$ is a unitary isomorphism from $\mathcal{M}L^2(\mathbb{R}^{n+1},\,d\tilde{\mu})$ onto $\mathcal{F}(X)\otimes \mathbb{C}_n $.

\begin{proof}[Proof of Theorem \ref{main02}] 
First, we show that $(1-D)^{-1}$ is an isometry. Let $ F \in \mathcal{M}L^2(\mathbb{R}^{n+1},\,d\tilde{\mu})$. By Corollary \ref{onb}, the orthogonality of $\{P_\beta\}$ implies 
\[ 
F = \sum_{\beta \in \mathbb{N}_0^n}\omega_{\beta}P_{\beta} \quad \text{and } \quad \|F\|^2 = \sum_{\beta \in \mathbb{N}_0^n} \beta! |\omega_{\beta}|^2
\]
where $\omega_\beta \in \mathbb{C}_n$ for each $\beta$ and the first sum converges in $L^2(\mathbb{R}^{n+1},\,d\tilde{\mu}) \otimes \mathbb{C}_n$ sense and also converges uniformly on compact sets by monogenicity of $F$. 
Since $P_\beta(x_0,\underline{x}) = e^{-x_0\underline{D}}\,\underline{x}^\beta$ and any partial differential operator commutes with $\underline{D}$ and hence with $e^{-x_0\underline{D}}$, it follows that
$\partial^{\alpha}P_{\beta}(0,\underline{0}) = \beta!\delta_{\alpha\beta}$. Thus 	
\[
|D^kF(0,\underline{0})|^2 = \sum_{|\beta| = k} \frac{1}{\beta!} |\partial^{\beta}F(0,\underline{0})|^2 = \sum_{|\beta| = k} \beta!\, |\omega_{\beta}|^2
\]	
Hence
\[
\|(1-D)^{-1}F\|^2 
= \sum_{k=0}^{\infty} |D^kF(0,\underline{0})|^2
= \sum_{\beta \in \mathbb{N}_0^n} \beta! |\omega_{\beta}|^2 
= \|F\|^2.
\]
This establishes the isometry of $(1-D)^{-1}$.
Next, we show that $(1-D)^{-1}$ is surjective. Let $\alpha \in \mathcal{F}(X)\otimes \mathbb{C}_n.$ Then $\alpha = \sum_{k=0}^{\infty} \alpha_k$ where $\alpha_k \in X^{\odot k}\otimes \mathbb{C}_n$ and $\|\alpha\|^2 =\sum_{k=0}^{\infty}  |\alpha_k|^2 < \infty.$ For each  $\underline{x} \in \R^n$, define $\exp_k(\underline{x}) \in (\R^n)^{\odot k}$ by
\begin{align*}
\exp_k(\underline{x}) \;
= \; \sum_{|\beta| = k} \frac{1}{\beta!}\, (x_1 e_1)^{\odot \beta_1}{\odot} \dots{\odot} (x_n e_n)^{\odot \beta_n} \;
= \; \sum_{|\beta| = k} \frac{\underline{x}^{\beta}}{\beta!}\, e^\beta.
\end{align*}
where $e^\beta = e_1^{\odot \beta_1}{\odot} \dots{\odot}  e_n^{\odot \beta_n}$. For each $k \in \N_0$, let
$f_k(\underline{x})= \alpha_k(\exp_k(\underline{x}))$ and $f(\underline{x}) = \sum_{k=0}^{\infty}	f_k(\underline{x})$.		
Then each $f_k$ is analytic, which implies $f(\underline{x})$ is analytic. Define $F(x_0,\underline{x}) = e^{-x_0\underline{D}}f(\underline{x})$ to be the C-K extension of $f$. By (\ref{Def of P_beta}), we have 
\[
F(x_0,\underline{x}) = \sum_{k=0}^{\infty} \sum_{|\beta| = k} \frac{1}{\beta!}  \alpha_k(e^\beta)\,P_\beta.
\]
It follows from Corollary \ref{onb} that
\begin{align*}
	\|F\|^2 = \sum_{k=0}^{\infty}\sum_{|\beta| = k}  \frac{1}{\beta!} |\alpha_k (e^{\beta})|^2 = \sum_{k=0}^{\infty}|\alpha_k|^2 = \|\alpha\|^2 < \infty.
\end{align*}
To show that $ (1-D)^{-1} F = \alpha$, let $m \in \mathbb{N}_0$. For $e^{\gamma} = e_1^{{\odot\gamma_1}}{\odot} e_2^{{\odot\gamma_2}}{\odot}\dots {\odot} e_n^{{\odot\gamma_n}}$ with $|\gamma| = m$, we have 
\begin{align*}
[D^{m}F(0,\underline{0})](e^{\gamma}) \;
&=\; \partial_{e_1}^{\gamma_1}\dots\partial_{e_n}^{\gamma_n}\Big[e^{-x_0\underline{D}}\Big(\sum_{k=0}^{\infty} \alpha_k(\exp_k(\underline{x}))\Big)\Big](0,\underline{0}) \\
&= \; \partial_{e_1}^{\gamma_1}\dots\partial_{e_n}^{\gamma_n}\Big[\sum_{k=0}^{\infty} \alpha_k\big(\exp_k(\underline{x})\big)\Big](\underline{0})\\
&= \; \sum_{k=0}^{\infty} \Big[\partial_{e_1}^{\gamma_1}\dots\partial_{e_n}^{\gamma_n}\alpha_k\big(\exp_k(\underline{x})\big)\Big](\underline{0}).
\end{align*}
We can differentiate term-by-term inside the power series because $F(0,\underline{x})$ is analytic. Since
\,$\partial_{e_1}^{\gamma_1}\dots\partial_{e_n}^{\gamma_n}(\underline{x}^{\beta})  = {\beta}!\,\delta_{\beta\gamma}$,
evaluating at $\underline{x} = \underline{0}$ gives		 
\[
\partial_{e_1}^{\gamma_1}\dots\partial_{e_n}^{\gamma_n}\big(\exp_k(\underline{x})\big) 
=
\begin{cases}
e_1^{{\odot\gamma_1}}{\odot} e_2^{{\odot\gamma_2}}{\odot}\dots {\odot} e_n^{{\odot\gamma_n}}, &\text{ for $k = m$} ; \\
0, & \text{ for $k \neq m$.}
\end{cases} 
\] 
It follows that $D^{m}F(0,\underline{0})(e^{\gamma}) = \alpha_{m}(e^{\gamma})$ for any $\gamma \in \mathbb{N}^n_0$ with $|\gamma| = m$. Since  $\{e^\gamma : \gamma \in \mathbb{N}^n_0 , |\gamma| = m\}$ forms an orthogonal basis for $(\R^n)^{\odot m}$, we conclude that $D^{m} F(0,\underline{0}) = \alpha_{m}$ and that
\[
(1-D)^{-1}F = \sum_{m=0}^\infty D^m F(0,\underline{0}) = \sum_{m=0}^\infty \alpha_m =\alpha. \qedhere 
\]

\end{proof}

\subsection*{Acknowledgments}
We would like to thank Brian Hall for his valuable suggestions and also to the anonymous reviewers for helpful comments. This work was partially supported by the Development and Promotion of Science and Technology Talents Project.

\subsection*{Declarations}

This publication is supported by multiple datasets, which are available at locations cited in the reference section.

	\label{References}
	
\end{document}